\documentclass[12pt,a4paper, reqno]{amsart}
\usepackage{amsfonts,amsmath,amssymb}
\usepackage{hyperref}
\usepackage{latexsym,fullpage,amsfonts,amssymb,amsmath,amscd,graphics,epic,enumerate}
\usepackage[all]{xy}
\usepackage{amssymb,amsthm,amsxtra,comment}
\usepackage{color}
\usepackage[dvipsnames]{xcolor}
\usepackage{amscd}
\usepackage{amsthm}
\usepackage{amsfonts,cancel}
\usepackage{amssymb}
\usepackage{url}
\usepackage{bbm}
\usepackage{wasysym}
\usepackage{MnSymbol}
\usepackage{kbordermatrix}
\usepackage{tikz}

\usepgflibrary{shapes.geometric}

\tikzstyle{V}=[draw, fill =black, circle, inner sep=0pt, minimum size=1.5pt]
\tikzstyle{C}=[draw, fill =white, circle, inner sep=0pt, minimum size=1.5pt]
\tikzstyle{over}=[draw=white,double=black,line width=2pt, double distance=.5pt]

\theoremstyle{plain}
\newtheorem*{theorem*}{Theorem}
\newtheorem*{remark*}{Remark}
\newtheorem*{example*}{Example}
\newtheorem{lemma}{Lemma}[subsection]
\newtheorem{proposition}[lemma]{Proposition}
\newtheorem{corollary}[lemma]{Corollary}
\newtheorem{theorem}[lemma]{Theorem}

\newtheorem*{conjecture*}{Conjecture}

\theoremstyle{definition}

\theoremstyle{remark}

 \newcommand{\op}{\operatorname}

\newcommand{\Hom}{\operatorname{Hom}}

\newcommand{\ch}{\operatorname{ch}}
\newcommand{\sch}{\operatorname{sch}}

\newcommand{\Ind}{\operatorname{Ind}}
\newcommand{\Coind}{\operatorname{Coind}}

\newcommand{\sdim}{\operatorname{sdim}}

\newcommand{\Span}{\operatorname{Span}}

\renewcommand{\Im}{\operatorname{Im}}

\newcommand{\sgn}{\operatorname{sgn}}

\renewcommand{\gg}{\mathfrak{g}}

\renewcommand{\dim}{\mathrm{dim}}

\newcommand{\finite}{\mathrm{finite}}
\newcommand{\Modulo}[1]{\ (\mathrm{mod}\ #1)}
\newcommand{\End}{\mathrm{End}}

\newcommand{\black}{\color{black}}

\newif\ifpaper
\papertrue

 \def\<{\langle}
  \def\>{\rangle}

\DeclareMathOperator{\tr}{tr}

\def\quotient#1#2{%
    \raise1ex\hbox{$#1$}\Big/\lower1ex\hbox{$#2$}%
}

\allowdisplaybreaks 
\begin{document}

\title{The Grothendieck ring of the periplectic Lie supergroup and supersymmetric functions}

\author{Mee Seong Im 
\and Shifra Reif
\and Vera Serganova} 
\address{Department of Mathematical Sciences, United States Military Academy, West Point, NY 10996 USA}
\email{meeseongim@gmail.com}
\address{Department of Mathematics, Bar-Ilan University, Ramat Gan, Israel}
\email{shifra.reif@biu.ac.il} 
\address{Department of Mathematics, University of California at Berkeley, Berkeley, CA 94720 USA}
\email{serganov@math.berkeley.edu} 
\date{\today}

\thanks{This project is partially supported by ISF Grant No. $1221/17$,  NSF grant $1701532$, 
and United States Military Academy's Faculty Research Fund. }

\begin{abstract}
We show that the Grothendieck ring of finite-dimensional representations of the periplectic Lie supergroup $P(n)$ is isomorphic to the ring of symmetric polynomials in $x_1^{\pm 1}, \ldots, x_n^{\pm 1}$  whose evaluation $x_1=x_2^{-1}=t$ is independent of $t$.  \\ 
 
\noindent  
 \textsc{Abstract.\hspace{-1.5mm}} 
Nous d\'emontrons que l'anneau de Grothendieck des repr\'esentations de dimension finie du super groupe de Lie p\'eriplectique $P(n)$ est isomorphe \`a l'anneau des polyn\^{o}mes sym\'etriques en $x_1^{\pm 1}, \ldots, x_n^{\pm 1}$ pour lesquels l'\'evaluation $x_1=x_2^{-1}=t$ est ind\'ependante de $t$.  
\end{abstract}

\maketitle
\bibliographystyle{amsalpha}  
\setcounter{tocdepth}{3}
\section{Introduction}\label{sec:intro} 
Supersymmetric polynomials are polynomials in $x_1,\ldots, x_m,y_1,\ldots,y_n$ which are doubly symmetric and satisfy the additional property that the evaluation $x_1=y_1=t$ is independent of $t$. 
Many properties of symmetric polynomials generalize to the supersymmetric case: 
It follows from \cite{STEMBRIDGE1985439} that the ring of supersymmetric polynomials is spanned by the supercharacters of $\mathfrak{gl}(m|n)$-modules which lie in the tensor algebra $T(V)$ of the natural representation $V$.  The super-Schur polynomials were shown in \cite{BR,Sergeev} to be supercharacters of simple $\mathfrak{gl}(m|n)$-modules that can be realized using Young tableaux lying in the $(m|n)$-hook.

For semisimple Lie algebras, the ring of characters is isomorphic to the Grothendieck ring, that is, the free abelian group on the objects of the category, modulo the relation
 $[B] = [A]+[C]$ 
 for every exact sequence 
 $0\rightarrow A \rightarrow B \rightarrow C  \rightarrow 0$. 
 \black
The ring structure is given by tensor products in the category. 
For the general linear group,  the Grothendieck ring 
$K[GL(n)]$ 
of the category of finite-dimensional representations   is isomorphic to the ring of symmetric Laurent polynomials
\[
K[GL(n)]\cong \mathbb Z\left[x_1^{\pm 1}, \ldots, x_n^{\pm 1} \right]^{S_n}
\]
(see for example, \cite[Sec. 23--24]{MR1153249}).
Moreover, 
Schur polynomials are images of irreducible representations under this isomorphism.

This description generalizes to all semisimple complex Lie algebras. 
In this case, 
the category admits complete reducibility 
and 
the characters of irreducible representations 
are given explicitly by 
the Weyl character formula.   
The  Grothendieck ring  is then isomorphic to the ring   
$\mathbb{Z}[P]^W$   
of   
$W$-invariants in the integral group ring 
$\mathbb{Z}[P]$, 
where 
$P$ 
is the corresponding weight lattice and 
$W$ 
is the Weyl group.
The isomorphism is given by the character map.

 For modules over Lie superalgebras, one has a parity shift functor  
 $\Pi$   
 which does not change the action of the Lie superalgebra.  
 Hence, 
 it is natural to consider one of the two quotients of the ring, 
 either by the relation  
 $[M] = [\Pi M] $  
 or  
 $[M] = -[\Pi M] $.  
 These quotients are isomorphic to the  ring of characters  and the  ring of supercharacters,   
 respectively.

Describing the ring of characters is harder in the super case. For the general linear Lie superalgebra, not all finite-dimensional modules lie in the tensor algebra $T(V)$. In general, the category of finite-dimensional modules over a simple Lie superalgebra is more subtle: 
it is not semisimple,   
characters (and supercharacters) of simple modules are not realized using Young tableaux 
and are hard to compute.    
Nevertheless, it was shown by  Sergeev  and  Veselov  in 
\cite{MR2776360},  
that for basic classical Lie superalgebras,  
the ring of supercharacter can be realized using supersymmetric functions.

 \black
 The theorem of  
 A.N. Sergeev and A.P. Veselov  
 states that the ring of supercharacters is equal to the subring of 
 $\mathbb{Z}[P]^W$, 
 admitting an extra condition which corresponds to the  isotropic roots  
 of the Lie superalgebra.  
In the case of the  general linear Lie supergroup 
$GL(m|n)$,  
the  ring of supercharacters  is isomorphic to the  ring of 
supersymmetric Laurent polynomials, 
namely, 
functions   
$f$  
in 
$\mathbb Z 
\left[ 
x_1^{\pm 1}, \ldots, x_m^{\pm 1}, y_1^{\pm 1}, \ldots, y_n^{\pm 1} 
\right]^{S_m\times S_n}$  for which  
$ f|_{x_1 = y_1 = t} $  
 is independent of 
 $t$.

 \black
The periplectic Lie superalgebra 
$\mathfrak p(n)$   
imposes further difficulties than  basic classical Lie superalgebras  
due to  the lack of 
an invariant bilinear form.    
It was only recently that  
its representations were understood   
and   
translation functors were computed in  
\cite{BDEHHILNSS}  
and  
\cite{BDEHHILNSS2}.

In this paper, 
we describe the  ring of supercharacters  of the  periplectic Lie supergroup 
$P(n)$. 
We show that  it is isomorphic to  the ring of supersymmetric functions  
with a suitable  supersymmetry condition,   
that is   
\begin{theorem}
\label{thm:reduced-GR-periplectic}
The ring of supercharacters of 
$P(n)$ 
is isomorphic to 
\[ 
J_n := 
 \{ f \in \mathbb{Z}[ x_1^{\pm 1}, \ldots, x_n^{\pm 1} ]^{S_n}: 
 f|_{x_1 = x_2^{-1} = t} \mbox{ is independent of } t  \}.
\] 
\end{theorem}

 The inclusion from left to right for 
 Theorem~\ref{thm:reduced-GR-periplectic} 
 is obtained by restriction to  rank-one  subalgebras  as done in 
 \cite[Prop.~4.3]{MR2776360}.
The other inclusion is much more involved.  
A key tool is the ring homomorphism   
$ds_n: J(P(n)) \rightarrow J(P(n-2))$ 
induced from the Duflo--Serganova functor. 
We use the realization of  
$ds_n$ 
as the evaluation map 
$f \mapsto f|_{x_n = x_{n-1}^{-1} = t}$, 
proven in  \cite{hoyt2016grothendieck} 
as well as the description of its kernel. 
The main step is to prove that 
$ds_n$   
is surjective  in order to apply  an  inductive argument. 
We construct  preimages of 
$ds_n$   
using Euler characteristics of  parabolic inductions   given in   
\cite{MR2734963}  
and  translation functors  given in 
\cite{BDEHHILNSS}.

\section{The periplectic Lie superalgebra and its representations}
\label{section:periplectic-salg}

\subsection{Lie superalgebras}

Given a $\mathbb{Z}_2$-graded vector superspace 
$V  =  V_{\overline 0}\oplus V_{\overline 1}$, 
the parity of a homogeneous (even) vector 
$v\in V_{\overline 0}$ 
is defined as 
$\overline v 
= \overline 0 \in \mathbb{Z}_2 
= \{ \overline 0 , \overline 1 \}$  
while the parity of an odd vector 
$v\in V_{\overline 1}$ 
is defined as 
$\overline v = \overline 1$. 
If the parity of a vector 
$v$ 
is 
$\overline 0$ 
or 
$\overline 1$, 
we say that 
$v$ has degree $0$ 
or 
$1$, 
respectively.  
We always assume that 
$v$  
is homogeneous whenever the notation  
$\overline v$   
appears in expressions. 
By 
$\Pi$ 
we denote the switch of parity functor.

The Lie superalgebra 
$\mathfrak{g} = \mathfrak{gl}(n|n)$ 
is defined to be the endomorphism algebra 
$\End( V_{\overline{0}} \oplus V_{\overline 1} )$, 
where 
$\dim V_{\overline 0} = \dim V_{\overline 1} = n$. 
Then 
$\mathfrak{g} 
= \mathfrak{g}_{\overline 0} \oplus \mathfrak{g}_{\overline 1}$,  
where  
\[ 
\mathfrak{g}_{\overline 0}  =  \End(V_{\overline 0})  \oplus  \End( V_{\overline 1} ) 
\qquad 
\mbox{ and } 
\qquad 
\mathfrak{g}_{\overline 1} 
= \Hom( V_{\overline 0}, V_{\overline 1} ) 
  \oplus 
  \Hom( V_{\overline 1}, V_{\overline 0} ). 
\] 
Let 
$[x,y] = xy - (-1)^{\overline x \overline y} yx $, 
where  $x$  and  $y$ 
are homogeneous elements of 
$\mathfrak{g}$, 
and extend 
$[\:\:\:,\:\:\:]$ 
linearly to all of 
$\mathfrak{g}$. 
By fixing a basis of 
$V_{\overline 0}$ 
and 
$V_{\overline 1}$, 
the superalgebra 
$\mathfrak{g}$ 
can be realized as the set of 
$2n\times 2n$  matrices, 
where  
\[ 
\mathfrak{g}_{\overline 0} = 
\left\{  
\begin{pmatrix}
A & 0 \\ 
0 & D \\ 
\end{pmatrix} : 
A, D\in M_{n,n}(\mathbb{C})
\right\}
\mbox{ and } 
\mathfrak{g}_{\overline 1} 
 = 
\left\{  
\begin{pmatrix} 
0 & B \\ 
C & 0 \\ 
\end{pmatrix} : 
B,C \in M_{n,n}(\mathbb{C})
\right\}, 
\] 
and  $M_{n,n}(\mathbb{C})$ 
are 
$n\times n$ 
complex matrices. 
Recall that the supertrace is defined by 
\[ 
\mathfrak{str} \begin{pmatrix} 
A & B \\ 
C & D   
\end{pmatrix}  
=   
\tr(A) - \tr(D). 
\] 

\subsubsection{Periplectic Lie superalgebra}

Let 
$V$ 
be an 
$(n|n)$-dimensional  vector superspace  equipped with a nondegenerate odd symmetric form
\begin{equation}
\label{eqn:bilinear-form-Pn}
\beta: V\otimes V \to \mathbb C, \quad  
\beta(v,w) = \beta(w,v), \quad 
\text{and} \quad  
\beta(v,w) = 0 \quad   
\text{if} \quad   
\overline {v} = \overline{w}.  
\end{equation}  
Then 
$\op{End}_{\mathbb C}(V)$ 
inherits the structure of a vector superspace from 
$V$.  
Let  
$\mathfrak{p}(n)$  
be the Lie superalgebra of all  
$X \in \operatorname{End}_{\mathbb C}(V)$ 
preserving  
$\beta$, i.e., 
$\beta$ 
satisfies the condition  
$$ 
\beta(Xv, w) + (-1)^{\overline{X} \overline{v}} \beta(v, Xw) = 0.
$$

With respect to a fixed bases for 
$V$, 
the matrix of 
$X \in  \mathfrak{p}(n)$ 
has the form 
$\left( 
\begin{smallmatrix}
A & B \\
C & -A^t
\end{smallmatrix} 
\right)$,  
where $A, B, C$  are 
$n\times n$ 
matrices such that  
$B$  
is symmetric and  
$C$  
is antisymmetric.

For the remainder of this manuscript, 
we will write 
$\mathfrak{g} :=  \mathfrak{p}(n)$.
  Note that 
  $\mathfrak{str}:\mathfrak{g} \rightarrow \mathbb{C}$  
  is a one-dimensional representation of  
  $\mathfrak{g}$. 
  We will also use the  
  $\mathbb Z$-grading  
  $\mathfrak{g}  
     = \mathfrak{g}_{-1} \oplus \mathfrak{g}_0 \oplus \mathfrak{g}_1$   
  where 
  $\mathfrak{g}_0  
     = \mathfrak{g}_{\overline 0} \simeq \mathfrak{gl}(n)$,  
  $\mathfrak{g}_{\pm 1} $  
  is the annihilator of  
  $V_{\overline 0}$  
  (respectively, $V_{\overline 1}$). 
  By $G$  
  we denote the algebraic supergroup  
  $P(n)$.

\subsection{Root systems}
\label{subsection:root-systems}
For the periplectic Lie superalgebra 
$\mathfrak{g}$, 
fix the standard Cartan subalgebra 
$\mathfrak{h}$ 
of diagonal matrices in 
$\mathfrak{g}_0$ 
with its standard dual basis 
$\{ \varepsilon_1, \ldots,  \varepsilon_n \}$.  
So we have a root space decomposition 
$\mathfrak{g} 
= \mathfrak{h}  \oplus  \left(  \bigoplus_{\alpha\in \Delta}  \mathfrak{g}_{\alpha} \right)$, 
where  
$\Delta = 
\Delta( \mathfrak{g}_{-1} ) \cup  
\Delta( \mathfrak{g}_0 ) \cup  
\Delta( \mathfrak{g}_1 )$,  
and   
\begin{align*} 
\Delta(\mathfrak{g}_0)  
    &= 
\{ 
\varepsilon_i - \varepsilon_j: 1\leq i \not= j \leq n
\},   \\
\Delta(\mathfrak{g}_1) 
     = 
\{ 
\varepsilon_i + \varepsilon_j : 1\leq i\le j \leq n&  
\}, \quad 
\mbox{ and }
\quad 
\Delta(\mathfrak{g}_{-1}) 
     =  
\{  
-(\varepsilon_i + \varepsilon_j) : 1 \leq i < j \leq n 
\}.  
\end{align*}

The set of simple roots is chosen to be 
\[
\Pi 
   = 
    \{ -2\varepsilon_1, \varepsilon_1-\varepsilon_2, \ldots, \varepsilon_{n-1}-\varepsilon_n \}. 
\]  
This implies that 
$
\Delta^+( \mathfrak{g}_{0} ) 
= 
\{ \varepsilon_i-\varepsilon_j : 1 \leq  i < j \leq n \}$. 
Our Borel subalgebra is then 
$\mathfrak b_0  \oplus  \mathfrak g_{-1}$, 
where 
$\mathfrak b_0 
= \bigoplus_{ \alpha\in\Delta^+( \mathfrak g_0 ) }   \mathfrak g_\alpha$ 
and 
$\mathfrak g_{-1} 
=  \bigoplus_{\alpha\in \Delta( \mathfrak g_{-1} )}   \mathfrak g_\alpha$.

Let 
\[ 
\mathcal{R}_{0}^{} 
=  \prod_{\alpha\in\Delta^+(\mathfrak{g}_{0})} ( 1-e^{-\alpha} ), 
\quad 
\mbox{ and }
\quad 
\mathcal{R}_{-1}^{} 
=  \prod_{\alpha\in\Delta(\mathfrak{g}_{-1})} ( 1-e^{-\alpha} ). 
\] 
For 
$W = S_n$, 
the Weyl group of the even subalgebra of 
$\mathfrak p(n)$, 
$\mathcal{R}_{-1}^{}$ 
is 
$W_{}$-invariant 
and 
$e^{\rho} \mathcal{R}_{0}^{}$ 
is 
$W$-anti-invariant.

\subsection{Weight spaces}
\label{subsection:weight-spaces }
 Let $\mathcal C_n$ 
 be the  category  of  finite-dimensional representation of 
 $\mathfrak {g}$ 
 and 
 $\mathcal F_n$ 
 be the category of finite-dimensional
  representation of 
  $G$. 
  Both are abelian symmetric rigid tensor categories. 
  The latter category is equivalent to the category of finite-dimensional 
  $\mathfrak{g}$-modules, 
  integrable over the underlying algebraic group 
  $G_0 = GL(n)$,  
  see \cite{S1}.  

The Cartan subalgebra 
$\mathfrak{h}$ 
is abelian, 
so it acts locally-finitely on a finite-dimensional 
$\frak g$-module 
$M$. 
This yields a decomposition of 
$M$ 
as a direct sum of generalized weight spaces 
$M  =  \oplus_{\lambda\in\mathfrak{h}^*}M_{\lambda}$ where 
$M_{\lambda}  =  \{ v\in M: (h-\lambda(h))^{m} v \mbox{ for all } h \in \mathfrak{h}\} \not= \{ 0 \} $ 
for some sufficiently large 
$m$.   
If 
$M$ 
is a $G$-module,  
it is semisimple over  
$\mathfrak{g}_0$  
and hence   
$\mathfrak{h}$  
acts diagonally on  
$M$.

Suppose that 
$M  =  \bigoplus_{ \mu \in \mathfrak h ^* } M_\mu $ 
is weight space decomposition of a 
$\mathfrak g$-module $M$. 
Define the character of $M$ as
\[
\ch(M) 
     :=  \sum_{\mu \in \mathfrak h^*}  \dim ( M_{\mu} ) e^{\mu}, 
\] 
while the supercharacter is defined as   
\[ 
\sch(M)  
     :=  \sum_{\mu \in \mathfrak h^*}  \sdim ( M_{\mu} ) e^{\mu}. 
\]

Weights of modules in the abelian category 
$\mathcal{F}_n$ 
of finite-dimensional representations of the periplectic Lie supergroup  
$P(n)$  
are denoted as  
\[
\lambda 
= (\lambda_1, \ldots, \lambda_n)  
= \sum_{1\leq i\leq n} \lambda_i \varepsilon_i, \qquad \lambda_i \in \mathbb{Z}.  
\]  
Define the parity of 
$\lambda$ 
as 
$p(\lambda) = \frac{1}{2}\sum_{1\leq i\leq n} \lambda_i\Modulo 2$  
if 
$\sum_{1\leq i\leq n} \lambda_i$  
is even and  
$p(\lambda) = \frac{1}{2}(\sum_{1\leq i\leq n} \lambda_i+1)\Modulo 2$  
if   
$\sum_{1\leq i\leq n} \lambda_i$ 
is odd. 
Note that the standard ordering  of the weights  for our choice of positive roots  is 
$ \varepsilon_i > \varepsilon_j $ when $i< j$ 
and 
$ \varepsilon_i < 0 $ 
for all $i$.

A weight  
$\lambda$  
is dominant if and only if 
$\lambda_1  \geq  \lambda_2  \geq  \ldots  \geq  \lambda_n$. 
We will denote 
$\Lambda_n$  
as the set of  dominant integral weights. 
Simple objects  in 
$\mathcal{F}_n$ 
(up to isomorphism and parity-switch) 
are parametrized by 
$\Lambda_n$. 
Denote by 
$L(\lambda)$  
the simple module with highest weight  
$\lambda$  
with respect to the Borel subalgebra  
$\mathfrak{b}_0\oplus \mathfrak{g}_{-1}$,   
where the parity is taken  such that  
the parity of the highest weight vector  is  
$p(\lambda)$.

\subsection{Thin Kac modules}  
\label{subsection:thin-Kac} 
Let 
$V(\lambda)$ 
be a simple 
$\mathfrak{g}_0$-module with highest weight 
$\lambda$ 
with respect to the fixed Borel 
$\mathfrak{b}_0$ of 
$\mathfrak{g}_0$. 
Given a dominant integral weight 
$\lambda$, 
the thin Kac module corresponding to 
$\lambda$ 
is   
\[
\nabla(\lambda) 
= \prod\!{}^{n(n-1)/2} 
\Ind_{\mathfrak{g}_0 \oplus \mathfrak{g}_1}^{\mathfrak{g}} V(\lambda - \gamma) 
\simeq 
\Coind_{\mathfrak{g}_0 \oplus \mathfrak{g}_1}^{\mathfrak{g}}  V(\lambda),  
\] 
where 
$\gamma 
= \sum_{\alpha\in\Delta(\gg_{-1})}\alpha 
= \sum_{i=1}^n(1-n)\varepsilon_i$.  
We will also write  
$\nabla_{P(n)}(\lambda)$  
to specify that the  thin Kac module  is a representation over the  algebraic supergroup  
$P(n)$.

Let
$
\rho: = \sum_{1\leq i\leq n} (n-i)\varepsilon_i. 
$ 
We have: 
  
\begin{lemma} 
\label{lem:superchar-thin-Kac} 
The supercharacter of the thin Kac module   
$\nabla(\lambda)$   
with weight   
$\lambda$  
is  
\begin{equation}
\label{eqn:superchar-thin-Kac-Pn}
\sch \nabla(\lambda) 
= (-1)^{p(\lambda)} 
\frac{ \mathcal{R}_{-1} }{
e^{\rho}\mathcal{R}_0} 
\sum_{w\in W}(-1)^{\ell(w)}e^{w(\lambda+ \rho)}.
\end{equation} 
\end{lemma}

\begin{proof} 
Since $\nabla(\lambda)\cong \bigwedge 
(\mathfrak{g}_{-1}^*)\otimes V(\lambda)$ as $\mathfrak h$-modules,  
the supercharacter of $\bigwedge(\mathfrak{g}_{-1}^*)$ 
and the character of $V(\lambda)$ are 
\[
\sch \bigwedge (\mathfrak{g}_{-1}^*) 
= 
\prod_{\beta\in\Delta(\mathfrak{g}_{-1})} ( 1 - e^{-\beta} )   
\quad   
\mbox{ and } 
\quad  
\ch V(\lambda) =   
\left( e^{\rho_0}  
\mathcal{R}_0  
\right)^{-1} 
\sum_{w\in W}(-1)^{\ell(w)}w(e^{\lambda+\rho_0}),     
\]   
respectively,  
where $\rho_0 = \frac{1}{2}\sum_{\alpha\in\Delta^+  (\mathfrak{g}_{0}) } \alpha$. 
Since $-\rho_0 + w\rho_0 = \rho - w\rho$, we obtain 
\eqref{eqn:superchar-thin-Kac-Pn}. 
\end{proof}

\subsection{Weight diagrams} 
Let $\{ \overline{\lambda}_1, \ldots, \overline{\lambda}_n\} \subseteq \mathbb{Z}$ 
be such that 
$\lambda + \rho  =  \sum_{i=1}^n \overline{\lambda}_i \varepsilon_i$.  
The weight diagram 
$d_{\lambda}$ 
corresponding to a dominant weight 
$\lambda$ 
is the labeling of the line of integers by symbols 
$\bullet$ and $\circ$, 
where $i$ has label $\bullet$ if 
$i\in \{ \overline{\lambda}_1, \ldots, \overline{\lambda}_n \}$, 
and 
$\circ$ 
if 
$i \not\in \{ \overline{\lambda}_1, \ldots, \overline{\lambda}_n \}$. 
For example, 
$ds_{0}$ is 
\[  
\xymatrix@-1pc{  
\ldots   &  \underset{-1}{\circ}   &  \underset{0}{\bullet}   & 
\underset{1}{\bullet}   &  \ldots   &  \underset{n-1}{\bullet} 
 & \underset{n}{\circ}   &  \underset{n+1}{\circ}   &  \ldots 
}
\] 
and 
$d_{-3  \varepsilon_n - \varepsilon_{n-1} }$   
is  
\[ 
\xymatrix@-1pc{ 
\ldots   &  \underset{-4}{\circ}   &  \underset{-3}{\bullet} 
 & \underset{-2}{\circ}   &  \underset{-1}{\circ}   & \underset{0}{\bullet}  & 
\underset{1}{\circ}   &   \underset{2}{\bullet}  &  \underset{3}{\bullet} 
 & \ldots  &  \underset{n-1}{\bullet}. 
} 
\]

Note that  
$\lambda \leq \mu$  
if and only if  
$\lambda_i \geq \mu_i$ 
for each $i$.  
In terms of weight diagrams,   
the $i$-th black ball in 
$d_{\lambda}$  
(counted from left)  
lies further to the right of the  
$i$-th black ball of  
$d_{\mu}$.

\subsection{The Grothendieck ring}
\label{subsubsection:reduced-Gr-Pn}
Let 
$K(P(n))$ 
be the Grothendieck ring of  
$\mathcal{F}_n$, 
and define 
\begin{equation}
\label{eqn:red-Groth-ring}
J(P(n))   =  K(P(n))/\langle [M] + [\Pi M] : M \in \mathcal{F}_n \rangle.  
\end{equation}
The ring 
$J(P(n))$ 
is isomorphic to the reduced Grothendieck ring 
$K(P(n))/ \langle [M] - [\Pi M] : M \in \mathcal{F}_n \rangle$,  
with the isomorphism given by 
$[L(\lambda)] \mapsto  (-1)^{p(\lambda)}[L(\lambda)]$, 
where  
$L(\lambda)$  
is the simple module of highest weight  
$\lambda$.

One may identify $J(P(n))$ as the ring of supercharacters as follows: 
let  
$\Lambda \subseteq \mathfrak{h}^*$  
be the abelian group of integral weights of  
$\mathfrak{g}_0$  
and  
$W$  
be the Weyl group of 
$\mathfrak{g}_0$. 
The supercharacter function  
$\sch: J(P(n))\rightarrow  \Span_ {\mathbb Z }\{e^\lambda : \lambda\in\Lambda\}$,   
sends 
$[M]\mapsto \sch(M)$. 
 Since isomorphic modules have the same supercharacter and 
 $\sch(M) = -\sch \Pi M$, 
 $\sch$ 
 is well-defined. 
 Furthermore, 
 $\sch$ 
 is injective since two irreducible modules have the same character if and only if they are isomorphic.

 Throughout this manuscript, 
 we will also write 
 $x_i := e^{\varepsilon_i}$ 
 for 
 $1\leq i\leq n$. 
The ring 
$$
J_n := 
\{f\in \mathbb{Z}[x_1^{\pm 1},\ldots, x_n^{\pm 1}]^{S_n}: f|_{x_i = x_j^{-1}=t} 
\mbox{ is independent of }t \mbox{ for }i\not=j \}
$$ 
is then identified with a subring of 
$\Span_ {\mathbb Z } \{ e^\lambda : \lambda\in\Lambda \}$. 

The following lemma is proved using restriction  to  subalgebras of the form 
$\mathfrak g_{-\alpha} \oplus  \mathfrak h  \oplus  \mathfrak g_\alpha$ 
where 
$\alpha$ 
is an odd root and 
$2\alpha$ 
is not a root.

\begin{lemma}[{\cite[Prop.~4.3]{MR2776360}}]
\label{lemma:Pn-superchar-containment}
We have 
$ J(P(n))\subseteq J_n. $
\end{lemma}

\subsection{Translation functors} 
\label{subsection:Translation-functors} 
Let  $\Theta' = -\otimes V:\mathcal{F}_n\rightarrow \mathcal{F}_n$  
be an endofunctor. 
Consider the involutive anti-automorphism 
$\sigma:\mathfrak{gl}(n|n)\rightarrow \mathfrak{gl}(n|n)$ 
defined as 
\[ 
\begin{pmatrix}
A  & B \\  
C  & D \\  
\end{pmatrix}^{\sigma} 
:= 
\begin{pmatrix}
-D^t  &   B^t \\  
-C^t  &  -A^t \\  
\end{pmatrix}.  
\] 
One can see that 
$\mathfrak{p}(n) = \{ x\in \mathfrak{gl}(n|n): x^{\sigma} = x\}$, 
and we set  
$\mathfrak{p}(n)^{\perp}  :=  \{ x \in \mathfrak{gl}(n|n) : x^{\sigma} = -x \}$. 

Since 
$\mathfrak{p}(n)$  
and  
$\mathfrak{p}(n)^{\perp}$  
form maximal isotropic subspaces with respect to the form 
$\mathfrak{str}XY$  
we obtain 
a nondegenerate bilinear  
$\mathfrak{p}(n)$-invariant pairing 
  $\langle \:\cdot\: , \:\cdot\: \rangle: \mathfrak{p}(n)\otimes \mathfrak{p}(n)^{\perp}\rightarrow \mathbb{C}$.  
Choose  
$\mathbb{Z}$-homogeneous bases 
$\{ X_i\}$   
in  
$\mathfrak{p}(n)$  
and  
$\{ X^i\}$  
in  
$\mathfrak{p}(n)^{\perp}$  
such that  
$\langle X^i , X_j \rangle = \delta_{ij}$.  
We define the fake Casimir element as  
\[ 
\Omega 
:= 
     2 \sum_{i=1}^n X_i \otimes X^i \in \mathfrak{p}(n)\otimes \mathfrak{p}(n)^{\perp} 
\subseteq \mathfrak{p}(n) \otimes \mathfrak{gl}(n|n). 
\]
 
Given a $\mathfrak{p}(n)$-module $M$, 
let 
$\Omega_M : M\otimes V\rightarrow M\otimes V$ 
be the linear map 
\[ 
\Omega_M(m\otimes v) 
= 2 \sum_{1\leq i \leq n} (-1)^{\overline{X}_i \overline{m}} X_i m \otimes X^i v,
\]  
where 
$m\in M$ 
and 
$v\in V$ 
are homogeneous. 
By \cite[Lemma~4.1.4]{BDEHHILNSS}, 
we see that $\Omega_M$ 
commutes with the action of 
$\mathfrak{p}(n)$  
on  
$M\otimes V$  
for any  
$\mathfrak{p}(n)$-module $M$. 

For $k\in \mathbb{C}$, 
define a functor 
$\Theta_k':\mathcal{F}_n\rightarrow \mathcal{F}_n$ 
as 
$\Theta' = -\otimes V$ 
followed by the projection onto 
the generalized $k$-eigenspace for 
$\Omega$, i.e., 
\[
\Theta_k'(M) := \bigcup_{m> 0} \ker(\Omega-k\: \text{Id})^m \Big|_{M\otimes V}. 
\] 
Since 
$\Theta_k'=0$  
if  
$k\notin\mathbb{Z}$, 
we set 
$\Theta' = \bigoplus_{k\in \mathbb{Z}} \Theta_k'$, 
and 
$\Theta_k := \prod^{k} \Theta_k'$  
when  
$k\in \mathbb{Z}$.  
The endofunctors  
 $\Theta_k$ of $\mathcal{F}_n$  
 for 
 $k\in \mathbb{Z}$ 
 are exact.

The following is \cite[Prop.~5.2.2]{BDEHHILNSS}.  

\begin{proposition}[Translation of thin Kac modules]  \label{smallKac} Let $k\in\mathbb{Z}$. Then 
\begin{enumerate}
 \item\label{item:right-trans-fnr} 
    $\Theta'_k \nabla(\lambda) = \nabla(\mu'')$ 
    if  $d_{\lambda}$ 
    looks as follows at positions 
    $k-1, k, k+1$, with $d_{\mu''}$ 
    displayed underneath:
 \begin{align*} 
  d_\lambda 
  = \xymatrix{  & \underset{k-1}{\bullet}  & \underset{k}{\bullet}  & \underset{k+1}{\circ} }  \\
 d_{\mu''} 
  = \xymatrix{  & \underset{k-1}{\bullet}  & \underset{k}{\circ}  & \underset{k+1}{\bullet} }
\end{align*}
\item\label{item:left-trans-fnr}  
    $\Theta'_k \nabla(\lambda)=\Pi\nabla(\mu')$ 
    if  $d_{\lambda}$ 
    looks as follows at positions 
    $k-1, k, k+1$, with $d_{\mu'}$ displayed underneath:
\begin{align*} 
 d_{\lambda} 
 = \xymatrix{  & \underset{k-1}{\circ}  & \underset{k}{\bullet}  & \underset{k+1}{\bullet}   }  \\
  d_{ \mu'} 
  = \xymatrix{  & \underset{k-1}{\bullet}  & \underset{k}{\circ}  & \underset{k+1}{\bullet}  }
\end{align*}
\item\label{item:SES-trans-fnr} 
    In case 
    $d_\lambda$ 
    looks locally at positions 
    $k-1,k,k+1$ 
    as below, 
    there is a short exact sequence
\[
0 \rightarrow \nabla(\mu'') \rightarrow \Theta'_k \nabla(\lambda) 
\rightarrow  \Pi\nabla(\mu') \rightarrow 0, 
\] 
where 
$d_{\mu'}$  
and  
$d_{\mu''}$  
are obtained from  
$d_\lambda$   
by moving one black ball away from position  
$k$  
(to position $k-1$, 
respectively, 
$k+1$) 
as follow:
\begin{align*} 
  d_\lambda 
     = \xymatrix{  & \underset{k-1}{\circ}  & \underset{k}{\bullet}  & \underset{k+1}{\circ}   }  \\  
   d_{\mu'} 
     = \xymatrix{  & \underset{k-1}{\bullet}  & \underset{k}{\circ}  & \underset{k+1}{\circ}  }\\  
    d_{\mu''} 
     = \xymatrix{  & \underset{k-1}{\circ}  & \underset{k}{\circ}  & \underset{k+1}{\bullet}      
     }  
\end{align*}  
\item    
    $\Theta'_k\nabla(\lambda)  =  0$ 
    in all other cases. 
\end{enumerate} 
\end{proposition} 
\subsection{Parabolic induction} 
\label{parabolic induction}
We consider the standard scalar product on 
$\mathfrak h^*$ 
such that 
$(\varepsilon_i,\varepsilon_j) = \delta_{i,j}$. 
Let 
$\gamma$ 
be some weight. 
Set
$$
\mathfrak{k} := \mathfrak h\oplus \bigoplus_{ ( \alpha, \gamma ) = 0} \mathfrak{g}_{\alpha}, 
\quad  
\mathfrak{r} := \bigoplus_{ ( \alpha, \gamma ) > 0 } \mathfrak{g}_{\alpha}, 
\quad   
\mathfrak{q} := \mathfrak{k} \oplus \mathfrak{r}.
$$
The subalgebra 
$\mathfrak{q}$ 
is called a parabolic subalgebra of 
$\mathfrak{g}$ 
with the Levi subalgebra  
$\mathfrak{k}$ 
and the nilpotent radical 
$\mathfrak{r}$.
If 
$Q$ 
is the corresponding parabolic subgroup of 
$G$, 
then 
$G/Q$ 
is a generalized flag supervariety.

Let 
$\lambda$ 
be a weight such that 
$\lambda(\mathfrak{h}\cap [\mathfrak{k}, \mathfrak{k}]) = 0$. 
We denote by 
$\mathcal{O}(-\lambda)$ 
the line bundle on 
$G/Q$ 
induced by the one dimensional representation of  
$Q$  
with weight  
$-\lambda$.  
Set 
\[ 
\mathcal{E}(\lambda)  =  \sum_i (-1)^i \sch H^i(G/P,\mathcal{O}(-\lambda)).  
\] 
By definition 
$\mathcal E(\lambda)$ 
is in 
$J(P(n))$. 
By  \cite[Prop.~1]{MR2734963}, 
\[  
\mathcal{E}(\lambda)  
 =  \frac{1}{  e^\rho 
\mathcal{R}_0} \sum_{w\in W}   
(-1)^w w 
\left( 
e^{\lambda + \rho} \prod_{\alpha\in \Delta_1(\mathfrak{r})} ( 1 - e^{-\alpha} )
\right) 
\] 
where 
$\Delta_1 (\mathfrak{r}) := \{ \alpha \in \Delta_{1} : ( \alpha, \gamma ) > 0  \}$.  
 
\section{Duflo--Serganova homomorphism for $P(n)$}
\label{section:DS-functor}
We show that the Duflo--Serganova functor  induces  a homomorphism between 
the rings of supercharacters, 
and discuss the kernel of this homomorphism 
(cf. \cite{hoyt2016grothendieck}).  

\subsection{The $ds_n$ homomorphism} 
\label{subsection:well-defined-kernel} 
Let 
$x\in \mathfrak{g}_1\oplus \mathfrak g_{-1}$ 
such that 
$[x, x] = 0$. 
Then 
$x^2$ 
is zero in 
$U( \mathfrak g )$ 
and for every 
$\mathfrak g$-module $M$, 
we define 
$
M_x = \ker_M x/ xM, 
$
and similarly,
$
\mathfrak g_x = \mathfrak g^x /[x, \mathfrak g], 
$ 
where 
$\mathfrak{g}^x = \{ g \in \mathfrak{g}: [x, g] = 0 \}$.
By \cite[Lemma~6.2]{duflo2005associated}, 
$M_x$  carries a natural  
$\mathfrak{g}_x$-module structure. 
Moreover, the
  Duflo--Serganova functor  
  $DS_x: M\mapsto M_x$
  is a symmetric monoidal functor from the category of 
  $\mathfrak{g}$-modules to the category of $\mathfrak{g}_x$-modules. 
We consider a special case of  the DS functor:  
\[  
DS_{n}: \mathcal{F}_n \rightarrow \mathcal{F}_{n-2},\quad DS_x(M) = M_x, 
\]
defined as follows.
Suppose 
$x = x_\beta$ for $\beta\in \Delta( \mathfrak g_{-1} )$. 
By \cite[Lemma~5.1.2]{ES-Deligne-perip}, 
$\mathfrak g_x \cong \mathfrak p(n-2)$. 
We embed  
$\mathfrak g_x$ 
in 
$\mathfrak g$ 
such that the Cartan subalgebra  
$\mathfrak h_x$ 
of  
$\mathfrak g_x$  
is contained in  
$\{h\in \mathfrak h : \beta(h)=0\}$.  
By  \cite[Sec.~3]{hoyt2016grothendieck},   
the map   
\begin{equation}  
\label{eqn:induced-ds-map}  
ds_{x}:  J(P(n)) \rightarrow  J(P(n-2)), 
\quad  
\mbox{ where } 
\:\: 
ds_{x}(\sch M)  =  \sch(DS_x(M)), 
\end{equation}  
 is well-defined and equal to  
$(\sch M)|_{h_{x}}$.  
Since  
$\beta  =  -\varepsilon_i - \varepsilon_j$  and 
$f\in J(P(n))$  satisfies that 
$f|_{x_i = x_j^{-1} = t}$   
 is independent of  
$t$,  
we get that 
$f|_{h_{x}} 
= f|_{ \{ h \in \mathfrak{h}: \varepsilon_i(h) = -\varepsilon_j(h)\} } = f|_{ x_i = x_j^{-1} }$.

Let $-\varepsilon_i - \varepsilon_j \in \Delta(\mathfrak{g}_{-1})$, 
and define  
\begin{equation}
ds_n: J(P(n)) \rightarrow J(P(n-2)),   \quad  ds_n  =  p \circ ds_{x_{\alpha}}, 
\end{equation} 
where  $p$  is a bijection 
$p: \{1,\ldots, n \}\setminus \{ i, j \} \rightarrow \{ 1, \ldots, n-2 \}$. 
Since the elements in  $J(P(n))$  are  $S_n$-invariant, 
$ds_n$ is independent of  $i$ and $j$.  
Moreover, the map  $ds_n$  extends naturally to   $J_n$  
by the evaluation 
$ds_n(f)  =  f|_{ x_{n-1} = x_n^{-1} = t }$ 
(eventually we show that 
$J(P(n)) = J_n$ 
but for now 
$J_n\subseteq J(P(n))$).

We define 
\begin{equation}
\label{eqn:dsnk}
ds_n^{(k)} := ds_{n-2k+2} \circ \cdots \circ ds_{n}. 
\end{equation} 
Note that applying  
$ds_n^{(k)}$  
is the same as applying  
$ds_x$ 
for 
$x$ 
of higher rank.

\subsection{The kernel of $ds_n$}
\label{subsection:kernel-ds-map} 
The following proposition is a straightforward generalization of  
\cite[Thm.~17]{hoyt2016grothendieck}. 
  
\begin{proposition}
\label{lemma:kernel-DS-Pn}
The kernel of $ds_{n}$ is spanned by the supercharacters of thin Kac modules. 
\end{proposition}
\begin{proof}
Suppose $f\in \ker ds_n$.  Then 
$f$ 
is divisible by 
$1-x_{n-1}x_n$.  
Since $f$ is $W$-invariant,   
$f$ 
is also divisible by 
$\prod_{i<j}(1-x_ix_j)$  
and hence  
$f = \prod_{i < j}( 1 - x_i x_j )g = \mathcal{R}_{-1} \cdot g,$   
where  $g$  is also  $W$-invariant.   
Write  $g$  
as a linear combinations of Schur functions 
\[   
g = \sum_{\lambda\in\mathfrak{h}^*}^{\finite}  a_{\lambda}  e^{-\rho} \mathcal{R}_0^{-1}\sum_{w\in W}
(\sgn w) w(e^{\lambda + \rho}).  
\]    
Thus 
$f 
    = \sum_{\lambda\in \mathfrak{h}^*}^{\finite} a_{\lambda} \nabla( \lambda )$.  
\end{proof}  
\subsection{Translation functors and $DS_n$} 
We will need the following statement, 
 see  
\cite[Cor.~3.0.2]{ES-KW-perip}.   
\begin{lemma} 
\label{commute}  
The functor 
$DS_n$ 
commutes with   translations functors 
$\Theta'_k$.  
  \end{lemma}

\begin{corollary} 
\label{lemma:translation-functor-image} 
If 
$[M]\in \Im ds_{n}$, then 
$[\Theta_i(M)]\in \Im ds_{n}$ for every translation functor $\Theta_i$. 
\end{corollary}

\begin{proof}
Suppose that
$DS_n(A) = M$ 
for some finite-dimensional 
$P(n)$-module 
$A$. 
By Lemma~\ref{commute}, 
we have 
$ds_n([\Theta_i(A)]) 
= [DS_{n}(\Theta_i(A))] 
= [\Theta_i (DS_n (A))]
= [\Theta_i(M)]$.
\end{proof}

\section{Surjectivity of the Duflo-Serganova homomorphism}
\label{section:main-thm-proof} 
As explained in Proposition \ref{lemma:kernel-DS-Pn}, the kernel of $ds_n$ is well-understood. To prove Theorem~\ref{thm:reduced-GR-periplectic}, we need to understand the image of $ds_n$ as well:
\black

\begin{theorem}\label{thm:ds-map-surj}
The map  $ds_n:J(P(n))\rightarrow J(P(n-2))$  is surjective.   
\end{theorem}

We prove Theorem~\ref{thm:ds-map-surj} in three steps.  
We first show that if $[\nabla(0)]$ is in the image of $ds_n^{(k)}$ for some $k\geq 0$  
(see \eqref{eqn:dsnk} for the definition), then $[\nabla(\mu)]$ is also in the image of $ds_n^{(k)}$   
for every $\mu$.   
We then show that $[\nabla(0)]$ is in the image of $ds_n$ by explicitly constructing its preimage.   
Finally, we show that the fact that $[\nabla(\mu)]$ is in the image of $ds_n^{(k)}$ for every $\mu$ implies that the map is surjective.   

The three steps are given in the following propositions. First we show:
\begin{proposition}
\label{prop:all-thin-Kac-image}
If $[\nabla(0)]\in \Im ds_n^{(k)}$ for some $k \geq 0$, then $[\nabla(\mu)]\in \Im ds_n^{(k)}$ for all $\mu\in \Lambda_{n}$. 
\end{proposition} 
This proposition is proved using the action of translation functors and Corollary \ref{lemma:translation-functor-image}. We make an inductive argument on the distance between the symbols $\bullet$ in the weight diagrams. Thus, we obtain that $[\nabla(\mu)]$ is in the image of for all $\mu\in \Lambda_{n}$.

Next we prove that:
 \begin{proposition}\label{nabla(0) is in the image}
One has $[\nabla(0)]\in \Im ds_{n}^{(k)}$. 
\end{proposition} 
 The preimage of $[\nabla(0)]$ is constructed using parabolic induction. We choose the parabolic subalgebra $\mathfrak{q}$ associated with $\gamma=-\sum_{l=2k+1}^n\varepsilon_l$ and take $\lambda = a(\varepsilon_1+ \ldots + \varepsilon_{2k})$ for $a\in \mathbb Z$ (see Section \ref{parabolic induction}). Then 
 \[ 
\mathcal{E}(\lambda) = 
\sum_{w\in S_n} w\left(
e^{\lambda}
\frac{\prod_{1\leq i\leq n, 2k < j \leq n, i< j} (1-x_i x_j)}{
\prod_{1\leq i<j\leq n} (1-x_i^{-1}x_j)
}
\right). 
\] 
and it is possible to show that it evaluates to $\sch \nabla(0)$ under the map $ds_n^k$.
 
Finally, we show tha:t
 \begin{proposition}
 \label{prop:surjectivity}
 If $\Span\{ [\nabla(\lambda)]: \lambda\in \Lambda_{n-2} \} \subseteq \Im ds_{n}$, then $J(P(n-2))\subseteq \Im ds_{n}$. 
 \end{proposition}

This is proved using an inductive argument on the associated graded algebra of $J(P(n-2))$ with respect to the filtration 
\begin{equation*}
0 = \ker ds_n^{(0)} \subseteq \ker ds_n^{(1)} \subseteq \ker ds_n^{(2)} \subseteq \ldots \subseteq \ker ds_n^{\left( \lfloor \frac{n}{2}\rfloor\right) } = J(P(n))  .
\end{equation*} 
We show that  $\ker ds_n^{(k)}/\ker ds_n^{(k-1)}$ is in the image of $ds_n$ by induction on $n$, using the fact that  $\ker ds_{n-2k}$ is in the image of  $ ds_{n-2k+2}$.  This follows since $\ker ds_{n-2k}$ is spanned by supercharacters of thin Kac modules by Proposition~\ref{lemma:kernel-DS-Pn} and supercharacters of thin Kac modules are in the image by   Proposition \ref{prop:surjectivity} and Proposition \ref{nabla(0) is in the image}. 

\subsection*{Proof of Theorem~\ref{thm:reduced-GR-periplectic}}\label{section:Groth-ring-superalgebra}
\label{subsection:main-result-proof}

By Lemma~\ref{lemma:Pn-superchar-containment}, 
we have that  $J(P(n))\subseteq  J_n$.   
Let us show the reverse inclusion. Suppose by induction that   $J(P(n-2))=J_{n-2}$.
By Theorem~\ref{thm:ds-map-surj}, 
the evaluation map 
\[ 
ds_n: J_n  
\rightarrow J(P(n-2))
\]  
given by $ds_n(f)=f|_{x_{n-1}=x_n^{-1}=t}$ is surjective when restricted to $J(P(n))$. 
Thus, every element of $J_n$  
is a sum of elements from $J(P(n))$ and $\ker {ds}_n$. 
By Proposition~\ref{lemma:kernel-DS-Pn},  $\ker {ds}_n\subseteq J(P(n))$ and   the claim follows. 
\qed

\bibliography{branching-diag}   
\end{document}